\documentclass[a4paper,12pt]{article}
\usepackage{}
\usepackage{fancyhdr,graphicx}
\usepackage{amsfonts}
\usepackage{amssymb}
\usepackage{amsthm}
\usepackage{newlfont}
\usepackage{amsmath}
\usepackage[top=2.5cm,bottom=2.5cm,left=2.5cm,right=2.5cm]{geometry}
\newtheorem{thm}{Theorem}[section]
\newtheorem{Lemma}[thm]{Lemma}

\usepackage{latexsym,bm}
\title{\vspace{0cm}{3-Factor-criticality of vertex-transitive graphs\thanks{This
work is supported by NSFC (grant no. 10831001)}} }
\author{Heping Zhang\footnote{Corresponding author.},  Wuyang Sun
\\\small{School of Mathematics and Statistics, Lanzhou University, Lanzhou, Gansu 730000, P. R. China}
\\\small{E-mail addresses: zhanghp@lzu.edu.cn, sunwy09@lzu.edu.cn}}

\date{}

\begin{document}
\pagestyle{plain}
\pagenumbering{arabic}

\maketitle \thispagestyle{empty}
\begin{abstract}
A graph of order $n$ is {\em $p$-factor-critical}, where $p$ is an integer of the same parity as $n$, if the removal of any set of $p$ vertices results in a graph with a perfect matching. 1-Factor-critical graphs and 2-factor-critical graphs are factor-critical graphs and bicritical graphs, respectively. It is well known that every connected vertex-transitive graph of odd order is factor-critical and every connected non-bipartite vertex-transitive graph of even order is bicritical. In this paper, we show that a simple connected vertex-transitive graph of odd order at least 5 is 3-factor-critical if and only if it is not a cycle.
\end{abstract}



\section{Introduction}
Only finite and simple graphs are considered in this paper. Let $G=(V(G),E(G))$ be a graph with vertex-set $V(G)$ and edge-set $E(G)$. The {\em order} of $G$ is the cardinality of $V(G)$. $G$ is called {\em odd} if its order is odd. A {\em perfect matching} of $G$ is a set of independent edges covering all the vertices in $G$. A graph $G$ with a perfect matching is {\em elementary} if its allowed edges (which are these edges contained in some perfect matchings of $G$) form a connected subgraph. The concepts of factor-critical and bicritical graphs were introduced by Gallai \cite{Gallai} and by Lov\'{a}sz \cite{lovasz}, respectively. A graph $G$ is called {\em factor-critical} if the removal of any vertex of $G$ results in a graph with a perfect matching. A graph is called {\em bicritical} if the removal of any pair of distinct vertices of $G$ results in a graph with a perfect matching. A graph $G$ is said to be {\em vertex-transitive} if for any two vertices $x$ and $y$ in $G$ there is an automorphism $\varphi$ of $G$ such that $y=\varphi(x)$. In \cite{plummer} (see Theorem 5.5.24 in Chapter 5), there is a following result.

\begin{thm}[\cite{plummer}]\label{1.1} If $G$ is a connected vertex-transitive graph of order $n$, then

(a) if $n$ is odd, $G$ is factor-critical, while

(b) if $n$ is even, $G$ is either elementary bipartite or bicritical.
\end{thm}

Favaron \cite{Favarvon} and Yu \cite{Yu} introduced, independently, a concept of $p$-factor-critical graphs, which is a generalization  of the concepts of factor-critical and bicritical graphs. A graph $G$ is said to be {\em $p$-factor-critical}, where $p$ is an integer of the same parity as $n$, if the removal of any set of $p$ vertices results in a graph with a perfect matching. A graph $G$ of even order $n$ is $q$-extendable \cite{M.D. Plummer}, where $q$ is an integer with $0\leq q<n/2$, if $G$ has a perfect matching and every set of $q$ independent edges is contained in a perfect matching of $G$. Favaron \cite{O. Favaron} showed that for $q$ even, every connected non-bipartite $q$-extendable graph is $q$-factor-critical. Two properties of $p$-factor-critical graphs are presented as follows.

Let $c_{0}(G)$ denotes the number of odd components of a graph $G$.

\begin{Lemma}[\cite{Favarvon,Yu}]\label{2.1} A graph $G$ is $p$-factor-critical if and only if $c_{0}(G-X)\leq|X|-p$ for every $X\subseteq V(G)$ with $|X|\geq p$.
\end{Lemma}

For a connected graph $G$, a {\em vertex-cut} of $G$ is a set of vertices whose removal disconnects $G$. The ({\em vertex-}){\em connectivity} of $G$, denoted by $\kappa(G)$, is the greatest integer $k$ such that $k$ is less than the order of $G$ and $G$ contains no vertex-cuts of $G$ with size less than $k$.  For  $F \subseteq E(G)$, $G-F$ denotes the resulting graph by removing the edges in $F$; $F$ is said to be an {\em edge-cut} of $G$ if $G-F$ is disconnected. The {\em edge-connectivity} $\lambda(G)$ of $G$ is the minimum cardinality over all edge-cuts of $G$. 

\begin{Lemma}[\cite{Favarvon}]\label{4} If a graph $G$ is $p$-factor-critical with $1\leq p<|V(G)|$, then $\kappa(G)\geq p$ and $\lambda(G)\geq p+1$.
\end{Lemma}

Theorem \ref{1.1} shows the factor-criticality and bicriticality of vertex-transitive graphs. A question arises naturally: what about $p$-factor-criticality of vertex-transitive graphs for $p\geq3$?

In this paper, we characterize the 3-factor-criticality of connected vertex-transitive odd graphs as follows.

\begin{thm}\label{100} A connected vertex-transitive odd graph of order at least 5 is 3-factor-critical if and only if it is not a cycle.
\end{thm}

To prove this, we apply the vertex-connectivity, edge-connectivity and several conditional edge-connectivities of vertex-transitive graphs. These results will be introduced in detail in Section 2. In Section 3, we will present some useful lemmas which play important roles in the proof of Theorem \ref{100}. In the last section, we prove Theorem \ref{100}.

\section{Conditional edge-connectivities of vertex-transitive graphs}
In this section, we will introduce some results on the connectivity, edge-connectivity, $s$-restricted edge-connectivity and cyclic edge-connectivity of vertex-transitive graphs.

Firstly, we present some notations. Let $X$ be a proper subset of $V(G)$ and set $\overline{X}$ = $V(G)\backslash X$. $\nabla(X)$ denotes the set of edges of $G$ with one end in $X$ and the other in $\overline{X}$. $N_{G}(X)=\{y: y\in\overline{X}$ such that there is some $x\in X$ with $xy\in E(G)\}$. $d_{G}(X)$ denotes the number of edges in $\nabla(X)$. $\nabla(\{v\})$, $N_{G}(\{v\})$ and $d_{G}(\{v\})$ are written as, in short, $\nabla(v)$, $N_{G}(v)$ and $d_{G}(v)$, respectively. The vertices in $N_{G}(v)$ are called {\em neighbors} of $v$. $G-X$ denotes the subgraph obtained by removing vertices in $X$ from $G$. $G[X]$ denotes the subgraph induced by $X$. If there is no confusion, then $N_{G}(X)$ and $d_{G}(X)$ are written as, in short, $N(X)$ and $d(X)$, respectively.

Watkins \cite{Watkins} studied the connection between connectivity and vertex-degree  for vertex-transitive graphs.

\begin{Lemma}[\cite{Watkins}]\label{2.2} Let $G$ be a connected $k$-regular vertex-transitive graph. Then $\kappa(G)>\frac{2}{3}k$.
\end{Lemma}

\begin{Lemma}[\cite{Watkins}]\label{2.3} If $G$ is vertex-transitive with degree $k$ = 4 or 6, then $\kappa(G)$ = $k$.
\end{Lemma}

It is well known that $\kappa(G)\leq\lambda(G)\leq\delta(G)$, where $\delta(G)$ is the minimum vertex-degree of $G$. A connected graph $G$ is said to be {\em maximally edge-connected} if $\lambda(G)=\delta(G)$. Mader \cite{Mader} proved the following result:

\begin{Lemma}[\cite{Mader}]\label{2.4} All connected vertex-transitive graphs are maximally edge-connected.
\end{Lemma}

A connected graph $G$ is said to be {\em super edge-connected}, in short, {\em super-$\lambda$}, if each of its minimum edge-cut is $\nabla(v)$ for some $v\in V(G)$. Tindell \cite{Tindell} characterized the super edge-connectivity of vertex-transitive graphs. An {\em imprimitive block} of $G$ is a proper non-empty subset $X$ of $V(G)$ such that for any automorphism $\varphi$ of $G$, either $\varphi(X)=X$ or $\varphi(X)\cap X=\emptyset$.  Tindell's result can be expressed as follows; See the two quotes \cite{Wang} and in \cite{DeVos}.

\begin{Lemma}[\cite{Tindell}]\label{2.5} A connected vertex-transitive graph $G$ with degree $k\geq3$ is super-$\lambda$ if and only if there is no imprimitive block of $G$ which is a clique of size $k$.
\end{Lemma}

For a connected graph $G$, an edge-cut $F$ of $G$ is said to be an {\em $s$-restricted edge-cut} if every component of $G-F$ has at least $s$ vertices, where $s$ is a positive integer. The {\em $s$-restricted edge-connectivity} of $G$ is the minimum cardinality over all $s$-restricted edge-cuts of $G$, denoted by $\lambda_{s}(G)$. A 2-restricted edge-cut and 2-restricted edge-connectivity are usually called a restricted edge-cut and restricted edge-connectivity, respectively.

Esfahanian and Hakimi \cite{Esfahanian} showed that if a connected graph $G$ of order $n\geq4$ is not a star $K_{1,n-1}$ then $\lambda_{2}(G)$ is well-defined and $\lambda_{2}(G)\leq\xi(G)$, where $\xi(G)$ is the minimum edge-degree of $G$, that is, the minimum number of edges adjacent to a certain edge in $G$. A connected graph $G$ is called to be {\em maximally restricted edge-connected}, if $\lambda_{2}(G)=\xi(G)$. Furthermore, a maximally restricted edge-connected graph $G$ is called to be {\em super restricted edge-connected}, in short, super-$\lambda_{2}$, if every minimum restricted edge-cut of $G$ isolates an edge, that is, every minimum restricted edge-cut of $G$ is a set of edges adjacent to a certain edge with minimum edge-degree in $G$. Xu \cite{Xu} studied restricted the edge-connectivity of connected vertex-transitive graphs.

\begin{Lemma}[\cite{Xu}]\label{2.6} Let $G$ be a connected vertex-transitive graph of order at least 4. Then $G$ is maximally restricted edge-connected if its order is odd or it has no triangle.
\end{Lemma}

Wang \cite{Wang} studied the super restricted edge-connectivity of connected vertex-transitive graphs. The {\em girth} of a graph $G$ with a cycle is the length of a shortest cycle of $G$.

\begin{Lemma}[\cite{Wang}]\label{2.7} If $G$ is a connected vertex-transitive graph with degree $k>2$ and girth $g>4$, then it is super-$\lambda_{2}$.
\end{Lemma}

We make an improvement on the previous result for vertex-transitive odd graphs.

\begin{thm}\label{2.8} If $G$ is a connected vertex-transitive odd graph with degree $k>2$ and girth $g>3$, then it is super-$\lambda_{2}$.
\end{thm}

Before we prove Theorem \ref{2.8}, we need to introduce some definitions and  some useful lemmas. A proper subset $X$ of $V(G)$ is called a {\em $\lambda_{2}$-fragment} of $G$ if $\nabla(X)$ is a minimum restricted edge-cut of $G$. A $\lambda_{2}$-fragment $X$ of $G$ is {\em trivial} if $|X|=2$. A nontrivial $\lambda_{2}$-fragment of $G$ with minimum cardinality is called a {\em $\lambda_{2}$-superatom} of $G$.

\begin{Lemma}[\cite{Godsil}]\label{2.9} Let $G$ be a vertex-transitive graph and $H$ be the subgraph of $G$ induced by an imprimitive block of $G$. Then $H$ is vertex-transitive.
\end{Lemma}

\begin{proof}[Proof of Theorem \ref{2.8}]
Since $G$ is regular by the vertex-transitivity of $G$ and $G$ is odd, $k$ is even. Hence $k\geq4$. By Lemmas \ref{2.4} and \ref{2.6}, $\lambda(G)=k$ and $\lambda_{2}(G)=2k-2$. We firstly claim that $d(S)\geq\lambda_{2}(G)$ for any subset $S\subseteq V(G)$ with $|S|\geq2$ and $|\overline{S}|\geq2$. In fact, if $G[S]$ or $G[\overline{S}]$ is disconnected, then $d(S)\geq2\lambda(G)=2k>2k-2=\lambda_{2}(G)$; if both $G[S]$ and $G[\overline{S}]$ are connected, then $\nabla(S)$ is a restricted edge-cut of $G$ and hence $d(S)\geq\lambda_{2}(G)$.

Suppose $G$ is not super-$\lambda_{2}$. Then $G$ has one $\lambda_{2}$-superatom and by the vertex-transitivity of $G$, it has at least two distinct $\lambda_{2}$-superatoms.

\textbf{Claim $\ast$.} $G$ has at least two distinct $\lambda_{2}$-superatoms $X$ and $Y$ such that $|X\cap Y|\geq1$.

Let $X$ be a $\lambda_{2}$-superatom of $G$. We have
\begin{align*}
|X|(|X|-1)&\geq\sum_{v\in X}d_{G[X]}(v)\\
&=k|X|-d(X)\\
&=k|X|-(2k-2)\\
&=|X|(|X|-1)-(|X|-k+1)(|X|-2).
\end{align*}
It follows that $|X|\geq k-1$ since $|X|>2$.

Suppose that every two distinct $\lambda_{2}$-superatoms of $G$ are disjoint. Then each $\lambda_{2}$-superatom is an imprimitive block of $G$. Thus $G[X]$ is vertex-transitive by Lemma \ref{2.9} and hence $G[X]$ is regular. Let $t$ be the regularity of $G[X]$. We have
$$2(k-1)=d(X)=|X|(k-t)\geq(k-1)(k-t),$$
i.e., $2\geq k-t\geq 1$, or $k-2\leq t\leq k-1$.

Since $X$ is an imprimitive block of $G$, $|X|$ is a divisor of $|V(G)|$. Noting that $|V(G)|$ is odd, $|X|$ is odd. It follows that $t$ is even. Then $t=k-2$. Noting that $2(k-1)=d(X)=|X|(k-t)=2|X|$, $|X|=k-1$. Hence $G[X]$ is a complete graph, contradicting $g>3$. Claim $\ast$ is proved.

Let $X$ and $Y$ be two distinct $\lambda_{2}$-superatoms of $G$ with $|X\cap Y|\geq1$. Such $X$ and $Y$ exist by Claim $\ast$. We will show that $|X\cap Y|\leq2$.

On the contrary, suppose that $|X\cap Y|\geq3$. Noting that $|X|\leq|V(G)|/2$ and $|Y|\leq|V(G)|/2$ by the definition of $\lambda_{2}$-superatom,
$$|\overline{X\cup Y}|=|V(G)|-|X|-|Y|+|X\cap Y|\geq|X\cap Y|\geq3.$$
We have that $d(X\cap Y)\geq\lambda_{2}(G)$ and $d(X\cup Y)\geq\lambda_{2}(G)$ by the claim in the first paragraph. Then, by the well-known submodular inequality (see, for example, p. 36-38 in \cite{Biggs}),
$$2\lambda_{2}(G)=d(X)+d(Y)\geq d(X\cap Y)+d(X\cup Y)\geq2\lambda_{2}(G).$$
It follows that $d(X\cap Y)=d(X\cup Y)=\lambda_{2}(G)$. Note that $\lambda_{2}(G)<2\lambda(G)$. We have that $\nabla(X\cap Y)$is a minimum restricted edge-cut with $3\leq|X\cap Y|<|X|$. This contradicts to the minimality of $\lambda_{2}$-superatom.

Next we will show that $|X\cap\overline{Y}|\leq2$. Suppose to the contrary that $|X\cap\overline{Y}|\geq3$. Then $|Y\cap\overline{X}|=|X\cap\overline{Y}|\geq3$. By the claim in the first paragraph, $d(X\cap\overline{Y})\geq\lambda_{2}(G)$ and $d(X\cup\overline{Y})\geq\lambda_{2}(G)$. Then a contradiction can be obtained by a similar argument as above.

So $3\leq|X|\leq4$. If $|X|=3$, then $G[X]$ is a path of length 2 since $G$ has no triangle and $G[X]$ is connected. Thus, $\lambda_{2}(G)=d(X)=3k-4>2k-2=\lambda_{2}(G)$, a contradiction. If $|X|=4$, then $|E(G[X])|=[4k-(2k-2)]/2=k+1\geq5$. It follows that $G[X]$ contains a triangle, a contradiction.
\end{proof}

Ou and Zhang \cite{Ou} studied the 3-restricted edge-connectivity of vertex-transitive graphs and proved the following results.

\begin{Lemma}[\cite{Ou}]\label{2.10} If $G$ is a connected $k$-regular vertex-transitive graph of order at least 6 and girth $g\geq4$, then either $\lambda_{3}(G)=3k-4$ or $\lambda_{3}(G)$ is a divisor of $|V(G)|$ such that $2k-2\leq\lambda_{3}(G)\leq3k-5$ unless $k=3$ and $g=4$.
\end{Lemma}

For a connected graph $G$, an edge-cut $F$ of $G$ is called a {\em cyclic edge-cut} if at least two components of $G-F$ contain cycles. The {\em cyclic edge-connectivity} of $G$ with a cyclic edge-cut is defined as the minimum cardinality over all cyclic edge-cuts of $G$, denoted by $\lambda_{c}(G)$. Let $\zeta(G)$= min\{$d(X)|X\subseteq V(G)$ and $X$ induces a shortest cycle in $G$\}. Wang and Zhang \cite{B. Wang} showed that $\lambda_{c}(G)\leq\zeta(G)$ for any graph $G$ with a cyclic edge-cut. If $\lambda_{c}(G)=\zeta(G)$, then $G$ is called {\em cyclically optimal}. Wang and Zhang \cite{B. Wang} found a sufficient condition for vertex-transitive graphs to be cyclically optimal.

\begin{Lemma}[\cite{B. Wang}]\label{2.11} Let $G$ be a connected vertex-transitive graph with degree $k\geq4$ and girth $g\geq5$. Then $G$ is cyclically optimal.
\end{Lemma}

\section{Some useful lemmas}

A subset $X$ of $V(E)$ is called an independent set of a graph $G$ if $E(G[X])=\emptyset$. The independent number of $G$ is the maximum cardinality of  independent sets of $G$, denoted by $\alpha(G)$.

\begin{Lemma}\label{3.1} Let $G$ be a connected vertex-transitive odd graph with degree $k\geq4$. Then $\alpha(G)<(|V(G)|-1)/2$.
\end{Lemma}
\begin{proof} On the contrary, suppose that there is an independent set $Y$ of $G$ such that $|Y|=(|V(G)|-1)/2$. Since $G$ is regular by the vertex-transitivity of $G$ and $G$ is odd, $G$ is non-bipartite. It follows that $G$ contains a odd cycle. Let $g_{0}$ be the length of a minimum odd cycle in $G$. Since $G$ is vertex-transitive, each vertex is contained in a constant number of the cycles of length $g_{0}$ in $G$, say this constant number is $m$. Let $n_{g_{0}}$ be the total number of odd cycles of length $g_{0}$ in $G$.

Since $Y$ is an independent set, each odd cycle of length $g_{0}$ contains at most $(g_{0}-1)/2$ vertices in $Y$ and at least $(g_{0}+1)/2$ vertices in $\overline{Y}$. Therefore, vertices in $Y$ are covered by all the cycles of length $g_{0}$ at most $\frac{1}{2}(g_{0}-1)n_{g_{0}}$ times and vertices in $\overline{Y}$ are covered by all the cycles of length $g_{0}$ at least $\frac{1}{2}(g_{0}+1)n_{g_{0}}$ times. On the other hand, we know that vertices in $Y$ and $\overline{Y}$ are exactly covered by all the cycles of length $g_{0}$  $m|Y|$ times and $m|\overline{Y}|$ times, respectively. Hence
\begin{align}
m|Y|\leq\frac{1}{2}(g_{0}-1)n_{g_{0}}
\end{align}
and
\begin{align}
m|\overline{Y}|\geq\frac{1}{2}(g_{0}+1)n_{g_{0}}.
\end{align}
Note that $|\overline{Y}|=|V(G)|-|Y|=|Y|+1$. We can obtain by (2)-(1) that $m\geq n_{g_{0}}$. Then $m=n_{g_{0}}$ since $m\leq n_{g_{0}}$. This means that each minimum odd cycle $C$ must contain all vertices in $G$ and hence $C$ is a hamiltonian cycle. Thus, every odd cycle in $G$ is a hamiltonian cycle. This is impossible. Because each hamiltonian cycle in $G$ has chords since $k\geq4$, we can find a smaller odd cycle in $G$, a contradiction.
\end{proof}

\begin{Lemma}\label{3.2} Let $G$ be a vertex-transitive graph with a triangle. Then, for each subset $X\subseteq V(G)$, the number of singletons in $G-X$ is not more than the number of edges in $G[X]$.
\end{Lemma}
\begin{proof}
Suppose to the contrary that there is a subset $X\subseteq V(G)$ such that the number of singletons of $G-X$ is more than the number of edges in $G[X]$. Let $t$ be the number of singletons in $G-X$, and $e(X)$ the number of edges in $G[X]$. Then $e(X)<t$. Since $G$ is vertex-transitive and contains a triangle, each vertex in $G$ is contained in a positive constant number $m$ of triangles. Since each triangle of $G$ containing a singleton of $G-X$ must pass through an edge in $G[X]$, there are at least $tm$ triangles passing through an edge in $G[X]$. Since $e(X)<t$, $G[X]$ has an edge $e$ which is contained in more than $m$ triangles. This means that more than $m$ triangles contain both ends of $e$, a contradiction.
\end{proof}

\begin{Lemma}\label{3.3} Let $G$ be a connected 4-regular vertex-transitive triangle-free odd graph. Then $G$ has no two distinct vertices $u$ and $v$ such that $N(u)=N(v)$.
\end{Lemma}
\begin{proof} We firstly show that there are no three distinct vertices $x$, $y$ and $z$ in $G$ such that $N(x)=N(y)=N(z)$. Suppose that $x$, $y$ and $z$ are such three distinct vertices in $G$ that $N(x)=N(y)=N(z)=\{a_{1},a_{2},a_{3},a_{4}\}$. Let $a_{5}$ be the neighbor of $a_{1}$ which is different from $x$, $y$ and $z$. By the vertex-transitivity of $G$, there are two distinct vertices $a'_{1}$ and $a''_{1}$ in $\{a_{2},a_{3},a_{4}\}$ such that $N(a_{1})=N(a'_{1})=N(a''_{2})$, say $a'_{1}=a_{2}$ and $a''_{1}=a_{3}$. Hence $a_{2}a_{5},a_{3}a_{5}\in E(G)$. Since $G$ is 4-regular and vertex-transitive, the fourth neighbor of $a_{4}$ must be $a_{5}$. This means that $G[\{x,y,z,a_{1},a_{2},a_{3},a_{4},a_{5}\}]$ is a component of $G$, a contradiction.

Suppose that $u$ and $v$ are two distinct vertices in $G$ such that $N(u)=N(v)=\{b_{1},b_{2},b_{3},b_{4}\}$. Since $G$ contains no triangle, $b_{1}$ has two neighbors in $V(G)\backslash(N(u)\cup\{u,v\})$, say $c_{1}$ and $c_{2}$. By the vertex-transitivity of $G$, there is a vertex $w\in \{b_{2},b_{3},b_{4}\}$ such that $N(w)=N(b_{1})$, say $w=b_{2}$. Now we will consider what will $N(b_{3})$ be. If $|N(b_{3})\cap\{c_{1},c_{2}\}|=2$, then $N(b_{1})=N(b_{2})=N(b_{3})$, a contradiction. If $|N(b_{3})\cap\{c_{1},c_{2}\}|=1$, say $N(b_{3})\cap\{c_{1},c_{2}\}=\{c_{1}\}$, then $N(b_{4})=N(b_{3})$ by the vertex-transitivity of $G$, which implies that $N(u)=N(v)=N(c_{1})$, a contradiction. So $N(b_{3})\cap\{c_{1},c_{2}\}=\emptyset$. Let $N(b_{3})=\{u,v,c_{3},c_{4}\}$. Then $N(b_{3})=N(b_{4})$ by the vertex-transitivity of $G$. Hence the graph showed in Fig. 1 is a subgraph of $G$.

\begin{figure}[h]
\begin{center}
\includegraphics[scale=0.7
]{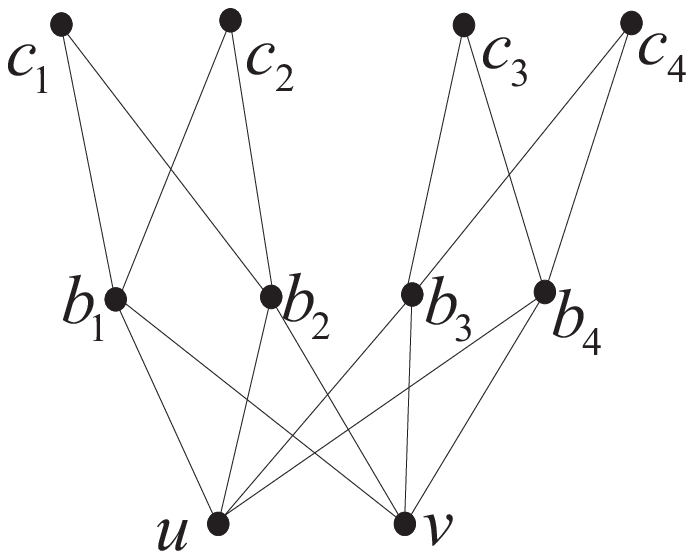}\\{\small{Fig. 1. A subgraph of $G$ with $N(u)=N(v)$.}}
\end{center}
\end{figure}

From Fig. 1, we can figure out that there are exactly 10 quadrangles containing $u$. By the vertex-transitivity of $G$, there are 10 quadrangles containing $v'$ for each vertex $v'\in V(G)$. Let $n_{4}$ be the number of quadrangles in $G$. Then $$4n_{4}=10|V(G)|.$$ It follows that $|V(G)|$ is even, a contradiction.
\end{proof}

\begin{Lemma}\label{3.4} Let $G$ be a connected vertex-transitive odd graph with degree $k=4$ and girth $g=4$. Then, for each edge $e$ in $G$, there are at least two distinct quadrangles containing $e$ and there is another edge $e'$ adjacent to $e$ such that the number of quadrangles containing $e'$ is the same as the number of quadrangles containing $e$.
\end{Lemma}
\begin{proof} Let $v$ be a vertex in $G$ incident to edges $e_{1}$, $e_{2}$, $e_{3}$ and $e_{4}$. Let $t_{i}$ be the number of quadrangles containing $e_{i}$ for $i=1,2,3,4$. By the vertex-transitivity of $G$, we only need to show that for each $i$, $t_{i}\geq2$ and there is an element $j\in\{1,2,3,4\}\backslash\{i\}$ such that $t_{j}=t_{i}$.

If there is an number $t_{j}$ such that $t_{j}\neq t_{i}$ for any $i\in\{1,2,3,4\}\backslash\{j\}$, then, by the vertex-transitivity of $G$, each vertex in $G$ is only incident to one edge contained exactly in $t_{j}$ quadrangles. This means that the set of edges contained exactly in $t_{j}$ quadrangles is a perfect matching of $G$, contradicting that $G$ is odd.

So either $t_{1}$, $t_{2}$, $t_{3}$ and $t_{4}$ have a common value, or two of them have a common value and the other two have another common value. Without loss of generality, we assume that $t_{1}=t_{2}$ and $t_{3}=t_{4}$.

Since $G$ is vertex-transitive and contains a quadrangle, each vertex is contained in a positive constant number of quadrangles. Let $n_{4}$ be the number of quadrangles in $G$ and $m$ be the number of the quadrangles containing $v$. Then $4n_{4}=m|V(G)|$. It follows that $4|m$ since $|V(G)|$ is odd. Hence $m\geq 4$. On the other hand, since each quadrangle containing $v$ passes through two edges incident to $v$, $2m=t_{1}+t_{2}+t_{3}+t_{4}=2t_{1}+2t_{3}$. Then $t_{1}+t_{3}=m\geq4$.

If $t_{1}=0$, then $t_{3}=m\geq4$. In this case, every quadrangle containing $v$ passes through $e_{3}$ and $e_{4}$. Since $G$ is 4-regular, $t_{3}=t_{4}\leq3$, a contradiction.

\begin{figure}[h]
\begin{center}
\includegraphics[scale=0.7
]{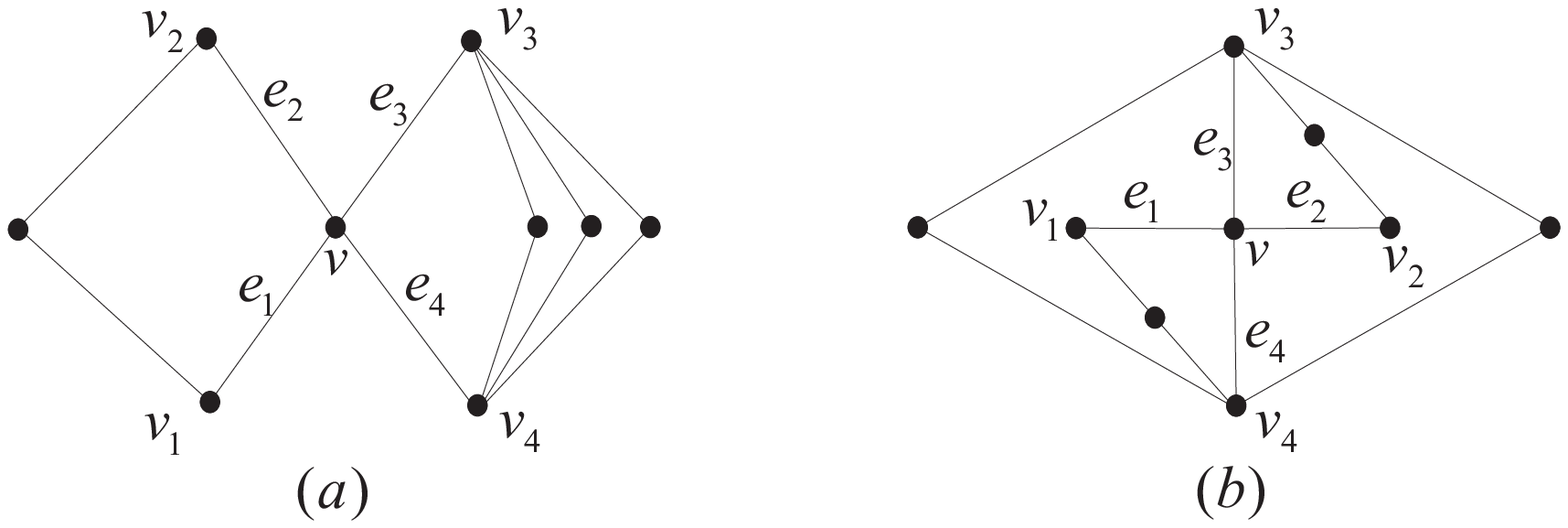}\\{\small{Fig. 2. The two possible subgraphs of $G$ with $t_{1}=1$ and $t_{3}\geq3$.}}
\end{center}
\end{figure}

If $t_{1}=1$, then $t_{3}=m-1\geq3$. If $e_{1}$ and $e_{2}$ are in a common quadrangle, then there are three quadrangles containing $e_{3}$ and $e_{4}$, see Fig. 2 ($a$). If there is no quadrangle containing $e_{1}$ and $e_{2}$, then the quadrangle containing $e_{1}$ is edge-disjoint from the quadrangle containing $e_{2}$ and there are two quadrangles containing $e_{3}$ and $e_{4}$ since $t_{3}=t_{4}\geq3$ and $t_{1}=t_{2}=1$, see Fig. 2 ($b$). In each case, we can see that $e_{3}$ and any of other three edges incident to $v_{3}$ are contained in a quadrangle. But there is no such edge incident to $v$ that it has this property. This means that there is no automorphism $\varphi$ of $G$ such that $\varphi(v)=v_{3}$, a contradiction.

So $t_{1}\geq2$. By a similar argument as above, $t_{3}\geq2$.
\end{proof}
\section{Proof of Theorem \ref{100}}
The ``only if'' part is trivial by Lemma \ref{4}. In this section we will mainly finish the ``if'' part.

Suppose that $G$ is not a cycle and is not 3-factor-critical. By Lemma \ref{2.1}, there is a set $X\subseteq V(G)$ with $|X|\geq3$ such that $c_{0}(G-X)>|X|-3$. By Theorem \ref{1.1}, $G$ is factor-critical. So $c_{0}(G-X)\leq|X|-1$ by Lemma \ref{2.1}. We have $$|X|-3<c_{0}(G-X)\leq|X|-1.$$ Since $c_{0}(G-X)$ and $|X|$ have different parity, $$c_{0}(G-X)=|X|-1.$$
Let $H_{1}$, $H_{2}$, $\ldots$, $H_{t}$ be the odd components of $G-X$ where $t=c_{0}(G-X)$.

Since $G$ is vertex-transitive, $G$ is regular. Let $k$ be the regularity of $G$. Noting that $G$ is odd and is not a cycle, $k$ is even and $k\geq4$. It follows that there is no imprimitive block of $G$ which is a clique of size $k$. Otherwise, $|V(G)|$ will be a multiple of $k$, contradicting that $|V(G)|$ is odd. Thus $G$ is super-$\lambda$ by Lemma \ref{2.5}.

\textbf{Claim 1.} Every component of $G-X$ is odd.

If $G-X$ has an even component $H_{0}$, then $d(V(H_{i}))\geq k$ for $0\leq i\leq t$ since $\lambda(G)=k$ by Lemma \ref{2.4}. Thus, $$k|X|=k(t+1)\leq\sum_{i=0}^{t}d(V(H_{i}))\leq d(X)\leq k|X|,$$
which implies that $d(V(H_{0}))=k$ and $X$ is an independent set of $G$. Hence $\nabla(V(H_{0}))$ isolates a vertex $v$ in $G$ since $G$ is super-$\lambda$, and $v\in X$. This means that $G[V(H_{0})\cup\{v\}]$ is a component of $G$, a contradiction. Claim 1 holds.

A graph $D$ is called {\em nontrivial} if $|V(D)|\geq2$. Let $g$ be the girth of $G$.

\textbf{Claim 2.} If $g\geq4$, then $G-X$ has exactly one nontrivial component $H$, and $d(V(H))=2k$.

Now we suppose that $g\geq4$. By Claim 1, $H_{1}$, $H_{2}$, $\ldots$, $H_{t}$ are all the components of $G-X$. Without loss of generality, we assume that $H_{1}$, $H_{2}$, $\ldots$, $H_{p}$ are nontrivial components and $H_{p+1}$, $H_{p+2}$, $\ldots$, $H_{t}$ are singletons. For $i=1,2,\dots,p$, $|V(H_{i})|\geq3$ and $|\overline{V(H_{i})}|\geq2|X|-2\geq4$. If $G[\overline{V(H_{i})}]$ is connected, then $\nabla(V(H_{i}))$ is a restricted edge-cut of $G$ and hence $d(V(H_{i}))>\lambda_{2}(G)=2k-2$ since $G$ is super-$\lambda_{2}$ by Theorem \ref{2.8}. If $G[\overline{V(H_{i})}]$ is disconnected, then $d(V(H_{i}))\geq2\lambda(G)=2k>2k-2$ since $\lambda(G)=k$ by Lemma \ref{2.4}. So $d(V(H_{i}))>2k-2$, for $i=1,2,\dots,p$. We have
$$p(2k-2)+k(t-p)<\sum_{i=1}^{p}d(V(H_{i}))+k(t-p)=\sum_{i=1}^{t}d(V(H_{i}))=d(X)\leq k|X|.$$
Note that $t=c_{0}(G-X)=|X|-1$. It follows that $p<\frac{k}{k-2}\leq2$ and $\sum_{i=1}^{p}d(V(H_{i}))\leq k(p+1)$.

If $p=0$, then $\overline{X}$ is an independent set of size $(|V(G)|-1)/2$ in $G$, which contradicts that $\alpha(G)<(|V(G)|-1)/2$ by Lemma \ref{3.1}.

So $p=1$.  Then $2k-2<d(V(H_{1}))\leq2k$. Since $d(V(H_{1}))=k|V(H_{1})|-2|E(H_{1})|$ is even, $d(V(H_{1}))=2k$. Claim 2 is proved.

\textbf{Claim 3.} If $g\geq4$, then $g=4$ and $k=4$.

Suppose that $g\geq4$. By Claims 1 and 2, $G-X$ has exactly one nontrivial component $H$ and $H$ satisfies $|V(H)|$ is odd and $d(V(H))=2k$. Hence $|V(G)|\geq|V(H)|+2|X|-2\geq7$. By Lemma \ref{2.10}, either $\lambda_{3}(G)=3k-4$ or $\lambda_{3}(G)$ is a divisor of $|V(G)|$. Since $k$ is even, $d(Y)=k|Y|-2|E(G[Y])|$ is even for any $Y\subseteq V(G)$. Hence $\lambda_{3}(G)$ is even. Then $\lambda_{3}(G)$ is not a divisor of $|V(G)|$ since $|V(G)|$ is odd. Thus $\lambda_{3}(G)=3k-4$.

Note that $\lambda(G)=k$ by Lemma \ref{2.4}, $\lambda_{2}(G)=2k-2$ by Theorem \ref{2.8}, $|V(H)|\geq3$ and $|\overline{V(H)}|=2|X|-2\geq4$. We claim that $G[\overline{V(H)}]$ is connected. Otherwise, if $G[\overline{V(H)}]$ has exactly two components, then $d(V(H))\geq\lambda(G)+\lambda_{2}(G)=3k-2>2k$, a contradiction; if $G[\overline{V(H)}]$ has at least three components, then $d(V(H))\geq3\lambda(G)=3k>2k$, a contradiction. Thus $\nabla(V(H))$ is a 3-restricted edge cut of $G$. Then
$$2k=d(V(H))\geq\lambda_{3}(G)=3k-4,$$
which implies that $k=4$.

Next we will show that $\nabla(V(H))$ is also a cyclic edge-cut of $G$. Note that $|V(H)|\geq3$ and $k=4$. Then $|E(H)|=\frac{1}{2}(k|V(H)|-2k)\geq|V(H)|-1$ and equality holds only for $|V(H)|=3$. If $|V(H)|=3$, then $X$ is an independent set of size $(|V(G)|-1)/2$ in $G$, which contradicts that $\alpha(G)<(|V(G)|-1)/2$ by Lemma \ref{3.1}. So $|V(H)|>3$ and $|E(H)|>|V(H)|-1$. It follows that $H$ contains a cycle. Since $|\overline{V(H)}|\geq4$, $|E(G[\overline{V(H)}])|=\frac{1}{2}(k|\overline{V(H)}|-2k)>|\overline{V(H)}|-1$. This implies that $G[\overline{V(H)}]$ also contains a cycle. Thus $\nabla(V(H))$ is a cyclic edge-cut of $G$.

If $g\geq5$, then, by Lemma \ref{2.11},
$$d(V(H))\geq\lambda_{c}(G)=(k-2)g\geq5k-10>2k,$$
a contradiction. So $g=4$. Claim 3 is proved.

By Claim 3, $g=3$, or $g=4$ and $k=4$. Next, we will distinguish the two cases to produce contradictions.

\textbf{Case 1.} $g=3$.

Let $t$ and $e(X)$ denote the number of singletons in $G-X$ and the number of edges in $G[X]$, respectively. We will show that $t>e(X)$. This would contradict Lemma \ref{3.2}.

Since $G$ is super-$\lambda$ and $k$ is even, $d(V(H_{i}))\geq k+2$, for any nontrivial component $H_{i}$ of $G-X$. Let $p'$ be the number of nontrivial components of $G-X$. We have
 $$p'(k+2)+k(t-p')\leq\sum_{i=1}^{t}d(V(H_{i}))=d(X)=k|X|-2e(X).$$
Note that $t=c_{0}(G-X)=|X|-1$. It follows that $p'\leq\frac{k}{2}-e(X)$.

Now we show that $|X|\geq k$. If $k$ = 4 or 6, then $\kappa(G)=k$ by Lemma \ref{2.3}. Hence $|X|\geq\kappa(G)=k$ since $X$ is a vertex-cut of $G$. Now assume that $k\geq8$. By Lemma \ref{2.2}, $|X|\geq\kappa(G)>\frac{2}{3}k$. Note that $c_{0}(G-X)=|X|-1$ and $p'\leq\frac{1}{2}k$. We have
$$t=c_{0}(G-X)-p'>\frac{2}{3}k-1-\frac{k}{2}=\frac{k}{6}-1>0,$$ which implies that there is a singleton $u$ in $G-X$. Hence $|X|\geq|N(u)|\geq k$.

Therefore, we have
$$t=c_{0}(G-X)-p'\geq k-1-(\frac{k}{2}-e(X))=e(X)+\frac{k}{2}-1\geq e(X)+1.$$

\textbf{Case 2.} $g=4$ and $k=4$.

By Claims 1 and 2, $G-X$ has exactly one nontrivial component $H$ and $H$ satisfies $|V(H)|$ is odd and $d(V(H))=2k$. It follows that $X$ is an independent set of $G$. From the argument in Claim 3, we know that $|V(H)|>3$ and $G[\overline{V(H)}]$ is connected.

\textbf{Claim 4.} $|\nabla(v)\cap\nabla(V(H))|\leq2$ for any $v\in V(G)$.

Noting that $G$ is 4-regular, $H$ and $G[\overline{V(H)}]$ are connected, we have that $|\nabla(v)\cap\nabla(V(H))|\leq3$ for any $v\in V(G)$. If there is $u'\in V(G)$ with $|\nabla(u')\cap\nabla(V(H))|=3$, then either $\nabla(V(H)\backslash\{u'\})$ or $\nabla(V(H)\cup\{u'\})$ is a restricted edge-cut of size $2k-2$ in $G$. Since $G$ is super-$\lambda_{2}$ by Theorem \ref{2.8}, either $|V(H)\backslash\{u'\}|=2$ or $|\overline{V(H)}\backslash\{u'\}|=2$, contradicting that $|V(H)|>3$ and $|\overline{V(H)}|=2|X|-2\geq4$. Claim 4 is proved.

By Claim 4, $|N(V(H))|\geq4$. If $|N(V(H))|=4$ and $|X|\geq6$, then let $H'=G[V(H)\cup N(V(H))]$ and $X'$ be the set of singletons in $G-X$. Then $|X'|=c_{0}(G-X)-1=|X|-2\geq4$, $c_{0}(G-X')=|X|-3=|X'|-1$ and $H'$ is the nontrivial component of $G-X'$ with $d(V(H'))=d(V(H)\cup N(V(H)))=4\times4-8=8=2k$. Take $H'$ as $H$ and $X'$ as $X$. Repeat this operation until $|N(V(H))|=4$ and $|X|\geq6$ do not hold.

So we assume that $|N(V(H))|>4$, or $|N(V(H))|=4$ and $|X|<6$.

If $|N(V(H))|=4$, then $|X|$ = 4 or 5. If $|X|=4$, then $G-X$ has two singletons $a$ and $b$ with $N(a)=N(b)=X$, which contradicts Lemma \ref{3.3}. If $|X|=5$, then $G-X$ has exactly three singletons, which form a vertex-cut of size 3 in $G$, contradicting that $\kappa(G)=k=4$ by Lemma \ref{2.3}.

So $|N(V(H))|>4$. Let $G'=G[\nabla(V(H))]$ be the induced subgraph by edges in $\nabla(V(H))$. Then there is a vertex $v_{0}\in V(G')\cap X$ with $d_{G'}(v_{0})=1$. By Claim 4, $d_{G'}(v)\leq2$, for any $v\in V(G')$. Hence every component of $G'$ is a path. Let $l$ be the length of the component $P$ of $G'$ passing through $v_{0}$.

If $l=1$, then there is no quadrangle in $G$ containing the edge in $P$, contradicting Lemma \ref{3.4}.

If $l=2$, let $P=v_{0}v_{1}v_{2}$ and $N(v_{1})=\{v_{0},v_{2},u_{1},u_{2}\}$. Then $v_{0},v_{2}\in X$ and $v_{1}$, $u_{1},u_{2}\in V(H)$. It follows that there is no quadrangle in $G$ containing an edge in $\{v_{0}v_{1},v_{1}v_{2}\}$ and an edge in $\{v_{1}u_{1},v_{1}u_{2}\}$. Note that there are two distinct quadrangles containing $v_{1}u_{1}$ by Lemma \ref{3.4}. Then the two quadrangles also contain $v_{1}u_{2}$ and cover three edges incident to $u_{1}$. But any two quadrangles containing $v_{1}$ cover two or four edges incident to $v_{1}$. This means that there is no automorphism $\varphi$ of $G$ such that $\varphi(v_{1})=u_{1}$, a contradiction.

So $l\geq3$. Let $P=v_{0}v_{1}v_{2}v_{3}\ldots v_{l}$ and $N(v_{0})=\{v_{1},w_{1},w_{2},w_{3}\}$. Let $s_{1}$, $s_{2}$, $s_{3}$ and $s_{4}$ be the numbers of quadrangles containing $v_{0}v_{1}$, $v_{0}w_{1}$, $v_{0}w_{2}$ and $v_{0}w_{3}$, respectively. Since each quadrangle containing $v_{0}v_{1}$ must pass through $v_{1}v_{2}$  and a vertex in $\overline{X}\backslash V(H)$ and $d_{G[\overline{V(H)}]}(v_{2})=2$, $s_{1}\leq2$. By Lemma \ref{3.4}, $s_{1}\geq2$. Thus $s_{1}=2$. Assume that the two quadrangles containing $v_{0}v_{1}$ are $v_{0}v_{1}v_{2}w_{1}v_{0}$ and $v_{0}v_{1}v_{2}w_{2}v_{0}$. Then, by Lemma \ref{3.4}, $s_{1}=s_{4}=2$ and $s_{2}=s_{3}>2$. Furthermore, let $n_{4}$ be the number of quadrangles in $G$ and $m$ be the number of quadrangles containing $v_{0}$. By the vertex-transitivity of $G$, $4n_{4}=m|V(G)|$. Hence $4|m$ since $|V(G)|$ is odd. Note that $2m=s_{1}+s_{2}+s_{3}+s_{4}=2s_{1}+2s_{2}$. Then $4|(s_{1}+s_{2})$. Thus $s_{2}\geq6$ since $s_{1}=2$ and $s_{2}>2$.

\begin{figure}[h]
\begin{center}
\includegraphics[scale=0.7
]{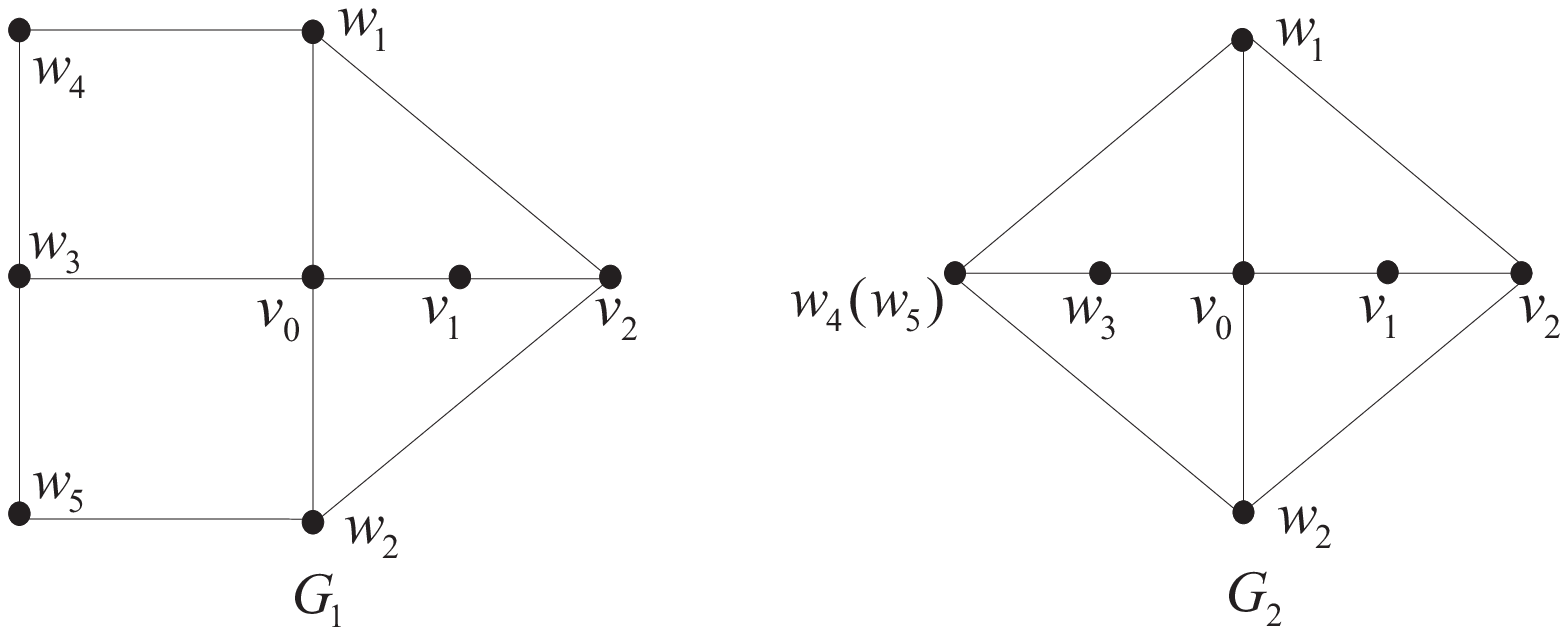}\\{\small{Fig. 3. The two possible induced subgraphs of $G$ with $l\geq3$.}}
\end{center}
\end{figure}

Let $v_{0}w_{1}w_{4}w_{3}v_{0}$ and $v_{0}w_{2}w_{5}w_{3}v_{0}$ be the two quadrangles containing $v_{0}w_{3}$. Note that $s_{1}=s_{4}=2$ and $g=4$. If $w_{4}\neq w_{5}$, then $G_{1}$ showed in Fig. 3 is an induced subgraph of $G$ by $\{v_{0},v_{1},v_{2},w_{1},w_{2},w_{3},w_{4},w_{5}\}$; if $w_{4}=w_{5}$, then $G_{2}$ showed in Fig. 3 is an induced subgraph of $G$ by $\{v_{0},v_{1},v_{2},w_{1},w_{2},w_{3},w_{4}\}$. It is easy to see that any quadrangle which contains $v_{0}$ and is not in $G_{i}$ must pass through $v_{0}w_{1}$ and $v_{0}w_{2}$. Noting that $G$ is 4-regular, the number of such quadrangles is at most 1. Hence $s_{2}\leq5$ since the numbers of quadrangles containing $v_{0}w_{1}$ in $G_{1}$ and in $G_{2}$ are, respectively, 3 and 4, a contradiction.


\begin{thebibliography}{6}
\bibitem{Biggs} N. Biggs, Algebraic Graph Theory, Cambridge University Press, Cambridge, 1993.
\bibitem{DeVos} M. DeVos, B. Mohar, Small separations in vertex-transitive graphs, Eletron. Notes Discrete Math. 24 (2006) 165-172.
\bibitem{Esfahanian} A. Esfahanian, S. Hakimi, On computing a conditional edge connectivity of graph, Inform. Process. Lett. 27 (1988) 195-199.
\bibitem{Favarvon} O. Favaron, On $k$-factor-critical graphs, Discuss. Math. Graph Theory 16 (1996) 41-51.
\bibitem{O. Favaron} O. Favaron, Extendability and factor-criticality, Discrete Math. 213 (2000) 115-122.
\bibitem{Gallai} T. Gallai, Neuer Beweis eines Tutte'schen Satzes, Magyar Tud. Akad. Mat. Kutat\'{o} Int. K\"{o}zl. 8 (1963) 135-139.
\bibitem{Godsil} C. Godsil, G. Royle, Algebraic Graph Theory, Springer-Verlag, New York, 2001.
\bibitem{lovasz} L. Lov\'{a}sz, On the stucture of factorizable graphs, Acta Math. Acad. Sci. Hungar. 23 (1972) 179-195.
\bibitem{plummer} L. Lov\'{a}sz, M.D. Plummer, Matching Theory, North-Holland, Amsterdam, 1986.
\bibitem{Mader} W. Mader, Minimale $n$-fach kantenzusammenh\"{a}ngenden Graphen, Math. Ann. 191 (1971) 21-28.
\bibitem{Ou} J. Ou, F. Zhang, 3-Restricted edge connectivity of vertex transitive graphs, Ars Combin. 74 (2005) 291-301.
\bibitem{M.D. Plummer} M.D. Plummer, On $n$-extendable graphs, Discrete Math. 31 (1980) 201-210.
\bibitem{Tindell} R. Tindell, Edge connectivity properties of symmetric graphs, Preprint, Stevens Institute of Technology, Hoboken, NJ, 1982.
\bibitem{Watkins} M.E. Watkins, Connectivity of transitive graphs, J. Combin. Theory 8 (1970) 23-29.
\bibitem{Wang} Y. Wang, Super restricted edge-connectivity of vertex-transitive graphs, Discrete Math. 289 (2004) 199-205.
\bibitem{B. Wang} B. Wang, Z. Zhang, On cyclic edge-connectivity of transitive graphs, Discrete Math. 309 (2009) 4555-4563.
\bibitem{Xu} J. Xu, Restricted edge-connectivity of vertex-transitive graphs, Chinese Ann. Math. Ser. A 21 (2000) 605-608.
\bibitem{Yu} Q. Yu, Characterizations of various matching extensions in graphs, Australas. J. Combin. 7 (1993) 55-64.
\end{thebibliography}
\end{document}